\documentclass [12pt]{article}
\usepackage{epsf}
\usepackage{psfrag}
\usepackage{ifpdf}
\usepackage{graphicx}
\usepackage{amsmath}
\title{On maximum Estrada indices of\\ bipartite graphs with some given
parameters\thanks{Supported by NSFC No.11371205 and PCSIRT.} }
\author{\small{Fei Huang, Xueliang Li, Shujing Wang}\\
{\small  Center for Combinatorics and LPMC-TJKLC}\\
{\small Nankai University, Tianjin 300071, China}\\
{\small Email: huangfei06@126.com; lxl@nankai.edu.cn;
wang06021@126.com} }
\date{}
\begin{document}

\newtheorem{theorem}{Theorem}[section]
\newtheorem{lemma}[theorem]{Lemma}
\newtheorem{coro}[theorem]{Corollary}
\newtheorem{theo}[theorem]{Theorem}
\newenvironment{proof}{\noindent {\bf
Proof.}}{\rule{3mm}{3mm}\par\medskip}

\maketitle

\begin{abstract}
The Estrada index of a graph $G$ is defined as
$EE(G)=\sum_{i=1}^ne^{\lambda_i}$, where $\lambda_1,$ $
\lambda_2,\ldots, \lambda_n$ are the eigenvalues of the adjacency
matrix of $G$. In this paper, we characterize the unique bipartite
graph with maximum Estrada index among bipartite graphs with given
matching number and given vertex-connectivity, edge-connectivity,
respectively.\\
\par\noindent{\bf Keywords:} \
Estrada index; walk; maximum matching; vertex cut;
connectivity\\
\par\noindent{\bf AMS Subject Classification:} \
05C50, 15A18, 05C35, 05C90, 92E10

\end{abstract}

\section{Introduction}

\hspace{8mm}Let $G$ be a simple graph on $n$ vertices. The
eigenvalues of $G$ are the eigenvalues of its adjacency matrix,
which are denoted by $\lambda_1\geq \lambda_2 \geq \ldots \geq
\lambda_n$. The Estrada index of $G$, put forward by Estrada
\cite{E1}, is defined as
$$EE(G)=\displaystyle\sum_{i=1}^ne^{\lambda_i}.$$
The Estrada index has multiple applications in a large variety of
problems, for example, it has been successfully employed to quantify
the degree of folding of long-chain molecules, especially proteins
\cite{E2,E4,E3}, and it is a useful tool to measure the centrality
of complex (reaction, metabolic, communication, social, etc.)
networks \cite{E5,E6}. There is also a connection between the
Estrada index and the extended atomic branching of molecules
\cite{E7}. Besides these applications, the Estrada index has also
been extensively studied in mathematics, see
\cite{Ili,LiJ,Zhang,Zhou,ZhouN}. Ili\'{c} and Stevanovi\'{c}
\cite{Ili} obtained the unique tree with minimum Estrada index among
the set of trees with a given maximum degree. Zhang, Zhou and Li
\cite{Zhang} determined the unique tree with maximum Estrada indices
among the set of trees with a given matching number. In \cite{DZ1},
Du and Zhou characterized the unique unicyclic graph with maximum
Estrada index. Wang et al. \cite{Wang} determined the unique graph
with maximum Estrada index among bicyclic graphs with fixed order,
and Zhu et al. \cite{Zz} determined the unique graph with maximum
Estrada index among tricyclic graphs with fixed order. More
mathematical properties on the Estrada index can be founded in
\cite{Gutman}.

A graph is bipartite if its vertex set can be partitioned into two
subsets $X$ and $Y$ so that every edge has one end in $X$ and the
other end in $Y$. We denote a bipartite graph G with bipartition
$(X, Y )$ by $G[X, Y ]$. If $G[X, Y ]$ is simple and every vertex in
$X$ is joined to every vertex in $Y$, then $G$ is called a complete
bipartite graph. Up to isomorphism, there is a unique complete
bipartite graph with parts of sizes $m$ and $n$, denoted $K_{m,n}$.
For an edge subset $A$ of the complement of $G$, we use $G + A$  to
denote the graph obtained from $G$ by adding the edges in $A$.

A matching in a graph is a set of pairwise nonadjacent edges. If $M$
is a matching, the two ends of each edge of $M$ are said to be
matched under $M$, and each vertex incident with an edge of $M$ is
said to be covered by $M$. A maximum matching is one which covers as
many vertices as possible. The number of edges in a maximum matching
of a graph $G$ is called the matching number of $G$ and denoted by
$\alpha'(G)$. Let $\mathcal{M}_{n,p}$ be the set of bipartite graphs
on $n$ vertices with $\alpha'(G)=p$.

A cut vertex(edge) of a graph is a vertex(edge) whose removal
increases the number of components of the graph. A(An) vertex(edge)
cut of a graph is a set of vertices(edges) whose removal disconnects
the graph. The connectivity(edge-connectivity) of a graph $G$ is
defined as
\[
\kappa(G) = \min\{|G|-1|, |S|: S \textrm{ is a vertex cut of } G\},\\
\kappa'(G) = \min\{|S|: S \textrm{ is an edge cut of } G\}.
\]
Let $\mathcal{C}_{n,s}$($\mathcal{D}_{n,s}$) denote the set of
bipartite graphs on $n$ vertices with $\kappa(G)=s$($\kappa'(G)=s$).
For other undefined terminology and notation we refer to Bondy and
Murty \cite{Bondy}.

In \cite{DZ2}, Du, Zhou and Xing determined the graphs with maximum
Estrada indices among graphs with given number of cut vertices,
connectivity, and edge connectivity, respectively. In this paper, we
consider bipartite graphs, and characterize the unique bipartite
graph with maximum Estrada indices among $\mathcal{M}_{n,p}$,
$\mathcal{C}_{n,s}$ and $\mathcal{D}_{n,s}$, respectively.

\section{Preliminaries}

\hspace{8mm}Denote by $M_k(G)$ the $k$-th spectral moment of a graph
$G$, i.e., $M_k(G)=\sum_{i=1}^n \lambda_i^k$. It is well-known
\cite{Deng} that $M_k(G)$ is equal to the number of closed walks of
length $k$ in $G$. Then
\begin{equation}\label{1}
EE(G)=\sum_{i=1}^n\sum_{k=0}^{\infty}\frac{\lambda_i^k}{k!}=\sum_{k=0}^{\infty}\frac{M_k(G)}{k!}.
\end{equation}
For $n$-vertex graphs $G_1$ and $G_2$, if $M_k(G_1)\leq M_k(G_2)$ for all
positive integers $k$, then by Eq.(1) we have that $EE(G_1)\leq
EE(G_2)$ with equality if and only if $M_k(G_1)= M_k(G_2)$ for all
positive integers $k$.

Let $k$ be a positive integer. For $u,v \in V(G)$, let $W_k(G; u,v)$
denote the set of $(u,v)$-walks of length $k$ in $G$, and let
$M_k(G; u,v) = |W_k(G; u,v)|$. For convenience, let $W_k(G; u) =
W_k(G; u,u)$ and $M_k(G; u)= M_k(G; u,u)$.

For graphs $G_1$ and $G_2$ with $u_1,v_1\in V(G_1)$ and $u_2,v_2\in
V(G_2)$, if $M_k(G_1; u_1,v_1) \leq M_k(G_2; u_2,v_2)$ for all
positive integers $k$, then we write$(G_1; u_1,v_1)\preceq (G_2;
u_2,v_2)$, and if $(G_1; u_1,v_1)\preceq (G_2; u_2,v_2)$ and there
is a positive integer $k_0$ such that $M_{k_0}(G_1; u_1,v_1)<
M_{k_0}(G_2; u_2,v_2)$, then we write $(G_1; u_1,v_1)\prec(G_2;
u_2,v_2)$. For convenience, we write $(G_1; u_1)\preceq(G_2; u_2)$
for $(G_1; u_1,u_1)\preceq(G_2; u_2,u_2)$, and $(G_1; u_1)\prec(G_2;
u_2)$ for $(G_1; u_1,u_1)\prec(G_2; u_2,u_2)$.

\begin{lemma}  \label{lem:2.1}
\cite{Gutman1} Let $G$ be a graph. Then for any edge $e\not\in
E(G)$, one has $EE(G+e)>EE(G).$
 \end{lemma}

\begin{lemma}\cite{Gutman}
If a graph $G$ is bipartite, and if $n_0$ is the nullity (=the
multiplicity of its eigenvalue zero) of $G$, then

\begin{equation}\label{0}
EE(G)=n_0+2\sum_{+}cosh(\lambda_i),
\end{equation}
where cosh stands for the hyperbolic cosine
$[cosh(x)=(e^x+e^{-x})/2]$, whereas $\sum_+$ denotes summation over
all positive eigenvalues of the corresponding graph.
\end{lemma}

As is well known \cite{Cvetkovi'c} that the spectrum of a complete
bipartite graph $K_{n_1,n_2}$ is $\sqrt{n_1n_2}, $ $-\sqrt{n_1n_2},$
$0$($n_1+n_2-2$ times).
 By the definition, we have
 \begin{lemma}   \label{lem:2.2} \cite{Gutman}
 $$
 EE(K_{n_1,n_2})=n_1+n_2-2+2cosh(\sqrt{n_1n_2}).
 $$
 \end{lemma}
 By the  monotonicity of $f(x)=cosh(x)$, it is obvious that
\begin{coro} \label{coro:2.3} \
\begin{equation}\label{1}
EE(K_{1,n-1})<EE(K_{2,n-2})<\ldots< EE(K_{\lfloor\frac{n}{2}\rfloor,\lceil\frac{n}{2}\rceil}).
\end{equation}
\end{coro}

\begin{lemma} \label{lem:2.4}
Let $G$ be a non-trivial graph with $u,v \in V(G)$ such that
$N_G(u)=N_G(v)$. Then for any $k\geq 0$, one has
$$
M_k(G; u)= M_k(G; v)=M_k(G; u, v)=M_k(G; v, u).
$$
\end{lemma}
\begin{proof}
For any walk $W\in W_k(G; u,u)$, let $f(W)$ be the walk obtained
from $W$ by replacing its first and last vertex  $u$ by $v$. This is
practical since $N_G(u)=N_G(v)$. Obviously, $f(W)\in W_k(G; v, v)$
and $f$ is a bijection from $W_k(G; u,u)$ to $W_k(G; v,v)$, and so
$M_k(G; u)= M_k(G; v)$. We can similarly construct a bijection from
$W_k(G; u,u)$ to $W_k(G; u,v)$ or $W_k(G; v,u)$. So we have
$$
M_k(G; u)= M_k(G; v)=M_k(G; u, v)=M_k(G; v, u),
$$
as desired.
\end{proof}

\begin{lemma} \label{lem:2.5}
Let $K_{n_1,n_2}$ be the complete bipartite graph with $X=\{x_1,
x_2,\ldots, x_{n_1}\}$ and $Y=\{y_1, y_2,\ldots, y_{n_2}\}$. For any
$k>0$, one has that for any $1\leq i,j \leq n_1$ and $1\leq r, s
\leq n_2$,

\begin{equation}\label{2}
M_{2k}(G; x_i, x_j)=n_1^{k-1}n_2^k \ , \ \ \ \ \ M_{2k}(G; y_r,
y_s)=n_2^{k-1}n_1^k.
\end{equation}\\
Furthermore,  $M_{2k}(G)=2(n_1n_2)^{k}$.
\end{lemma}
\begin{proof} Let $W=u_1(=x_i)u_2\ldots u_{2k}u_{2k+1}(=x_j)\in W_{2k}(G; x_i, x_j)$
be an $(x_i,x_j)$-walk of length $2k$. Since $G$ is a complete
bipartite graph, it is straightforward that $u_{2r+1}\in \{x_1, x_2,
\ldots, x_{n_1}\}$ and  $u_{2r}\in \{y_1, y_2, \ldots, y_{n_2}\}$
for $r=1,2,\ldots (k-1)$. Moreover, we know that each $u_{2r-1}$ can
be arbitrarily chosen from $X$ and each $u_{2r}$ can be arbitrarily
chosen from $Y$. Hence, for fixed $x_i$ and $x_j$ there are
$n_1^{k-1}n_2^k $ walks of length $2k$ between them,  that is,
$M_{2k}(G; x_i, x_j)=n_1^{k-1}n_2^k$ for any $1\leq i,j \leq n_1$.
Similarly, we can obtain $M_{2k}(G; y_t, y_r)=n_2^{k-1}n_1^k$ for
any $1\leq t, r \leq n_2$. By the definition of $W_{2k}(G)$, we have
$$
M_{2k}(G)=\sum_{i=1}^{n_1} M_{2k}(G; x_i)+\sum_{j=1}^{n_2}M_{2k}(G; y_j)=2(n_1n_2)^k.
$$
The proof is complete.
\end{proof}

Let $S_1=\{v_1, v_2, \ldots, v_s\}$ be an independent set of $G_1$
and $S_2=\{u_1, u_2,\dots, u_s\}$ an independent set of $G_2$. We
denote $G_1\cup_sG_2$ as the graph obtained from $G_1$ and $G_2$ by
identifying $v_i$ with $u_i$ for each $i$ ($1\leq i\leq s$). We
denote the  identified vertex set in $G_1\cup_sG_2$  by $S$.
Likewise, we can also get $G_1'\cup_sG_2'$ from $G_1'$ and $G_2'$,
where the two independent sets that should be identified are
$S_1'=\{v_1', v_2., \dots, v_s'\}$  and $S_2'=\{u_1', u_2., \dots,
u_s'\}$, respectively.
\begin{lemma}\label{lem:2.6}
Let $G=G_1\cup_sG_2$ and $G'=G_1'\cup_sG_2'$ be the graphs of order
$n$ defined as above satisfying the following conditions:
\begin{enumerate}
  \item For any $k>0$,
  \begin{equation}\label{0.1}
    M_k(G_1)\leq M_k(G_1') \ , \ \ \ \ M_k(G_2)\leq M_k(G_2');
  \end{equation}
    \item For any $1\leq i, j\leq s,$
    \begin{equation}\label{0.2}
    (G_1; v_i, v_j)\preceq (G_1'; v_i', v_j') \ , \ \ \ \ (G_2; u_i, u_j)\preceq (G_2'; u_i', u_j').
  \end{equation}
\end{enumerate}
Then for any $k>0$, $M_k(G)\leq M_k(G')$. Furthermore, $EE(G)\leq
EE(G')$, with equality holds if and only if all the equalities in
(\ref{0.1}) and (\ref{0.2}) hold.
\end{lemma}
\begin{proof}
For any $k>0$, let $W_{k}(G)$ denote the set of closed walks of
length $k$ in $G$, we can see that

\begin{equation}\label{0.3}
W_{k}(G)=W_{k}(G_1)\cup W_{k}(G_2)\cup W_{k}^3(G),
\end{equation}
where $W_{k}^3(G)$ is the set of closed walks of length $k$ in $G$
containing both vertices in $G_1\setminus S_1$ and vertices in
$G_2\setminus S_2$. Similarly, one has
\begin{equation}\label{0.4}
W_{k}(G')=W_{k}(G_1')\cup W_{k}(G_2')\cup W_{k}^3(G'),
\end{equation}
where $W_{k}^3(G')$ is the set of closed walks of length $k$ in $G'$
containing both vertices in $G_1'\setminus S_1'$ and vertices in
$G_2'\setminus S_2'$.

By (\ref{0.1}), we know that $|W_{k}(G_1)|\leq |W_{k}(G_1')|$ and
$|W_{k}(G_2)|\leq |W_{k}(G_2')|$. We only need to show that
$|W_{k}^3(G)|\leq |W_{k}^3(G')|$. In fact, there exists an
injection from $W_{k}^3(G)$ to $W_{k}^3(G')$. In the following, we
will construct such an injection.

For any $1\leq i,j\leq s$, by (\ref{0.2}) we know that for any
$l>0$,
\begin{equation}\label{0.8}
  M_{l}(G_1; v_i, v_j)\leq M_{l}(G_1'; v_i', v_j') \ , \ \ \ M_{l}(G_2; u_i, u_j) \leq M_{l}(G_2'; u_i', u_j').
\end{equation}
So there exist an injection $f^l_{i,j}$ from $W_{l}(G_1; v_i, v_j)$
to $W_{l}(G_1'; v_i', v_j')$, and an injection $g^l_{i,j}$ from
$W_{l}(G_2; u_i, u_j)$ to $W_{l}(G_2'; u_i', u_j')$ for any $1\leq
i,j\leq s$ and any $l> 0$, We will omit the subscript of $f^l_{i,j}$
and $g^l_{i,j}$ if there is no confusion about the first and last
vertices of the walks we considered.

For any $W\in W_{k}^3(G)$, we call a maximal $G_1$ walk of $W$ a
1-block, and a maximal $G_2$ walk of $W$ a 2-block. From the
definition, we have that the ends of a 1-block and a 2-block are
both contained in $S$. Since $W_{k}^3(G)$ is the set of closed walks
of length $k$ in $G_1$ and contains both vertices in $G_1\setminus
S_1$ and vertices in $G_2\setminus S_2$, there exist at least one
1-block and one 2-block, and the 1-blocks and 2-blocks appear one by
one alternately with equal number. Hence we can decompose $W$ as
follows:
$$
W=(B_0)B_1B_2B_3B_4\ldots B_r,\emph{where $r$ is even(odd)},
$$
where $B_{2i-1}$ is a 1-block of length $l_{2i-1}$, and $B_{2i}$ is
a 2-block of length $l_{2i}$. We define a map $\varphi$ from $W\in
W_{k}^3(G)$ to $W\in W_{k}^3(G')$ as follows:
$$
\varphi(W)=(g^{l_0}(B_0))f^{l_1}(B_1)g^{l_{2}}(B_2)f^{l_3}(B_3)g^{l_{4}}(B_4)\ldots.
$$
Then $\varphi(W)$ is a closed walk in $W\in W_{k}^3(G')$. Since both
$f^l_{i,j}$ and $g^l_{i,j}$  are injection, we can easily deduce
that $\varphi$ is an injection. Thus, $|W_{k}^3(G)|\leq
|W_{k}^3(G')|$, with equality holds if and only if for any $1\leq i,
j\leq s$ and any $l> 0$, $f^l_{i,j}$ and $g^l_{i,j}$ are bijections,
that is, all the qualities in (\ref{0.2}) hold. Hence, we have
\begin{eqnarray*}
M_k(G)&=&|W_{k}(G_1)|+|W_{k}(G_2)|+|W_{k}^3(G)|\\
 &\leq &|W_{k}(G'_1)|+|W_{k}(G'_2)|+|W_{k}^3(G')|\\
 &= &  M_k(G').
\end{eqnarray*}

Therefore, the result follows.
\end{proof}

\section{Maximum Estrada index of bipartite graphs with a given matching number}

\hspace{8mm} A covering of a graph $G$ is a vertex subset
$K\subseteq V(G)$ such that each edge of $G$ has at least one end in
the set $K$. The number of vertices in a minimum covering of a graph
$G$ is called the covering number of $G$ and denoted by $\beta(G)$.

\begin{lemma}   \emph{(The K\"{o}nig-Egerv\'{a}ry Theorem, \cite{Eger,K}).}
In any bipartite graph, the number of edges in a maximum matching is equal to
the number of vertices in a minimum covering.
\end{lemma}

Let $G=G[X,Y]$ be a bipartite graph such that $G\in
\mathcal{M}_{n,p}$. From Lemma 3.1, we know that $\beta(G)=p$. Let
$S$ be  a minimum covering of $G$ and  $X_1=S \cap X$, $Y_1=S \cap
Y$. Without  loss of generality, suppose that $|X_1|\geq |Y_1|$ in
the following analysis. Set $X_2=X\setminus X_1$, $Y_2=Y\setminus
Y_1$. We have that $E(X_2,Y_2)=\emptyset$ since $S$ is a covering of
$G$.

Let $G^*[X,Y]$ be a bipartite graph with the same vertex set as $G$
such that $E(G^*)=\{xy:(x\in X_1, y\in Y)\  or \ (x\in X_2, y\in
Y_1)\}$. Obviously, $G$ is a subgraph of $G^*$. From Lemma 2.1, we
know that

\begin{equation}\label{1.1}
EE(G)\leq EE(G^*),
\end{equation}
with equality holds if and only if $G\cong G^*$.
Let $$G^{**}=G^*-\{uv: u\in X_2, v\in Y_1\}+\{uw: u\in X_2, w\in X_1\}, $$
Then we have the following conclusion:

\begin{figure}[h,t,b,p]
\begin{center}
\scalebox{1}[1]{\includegraphics{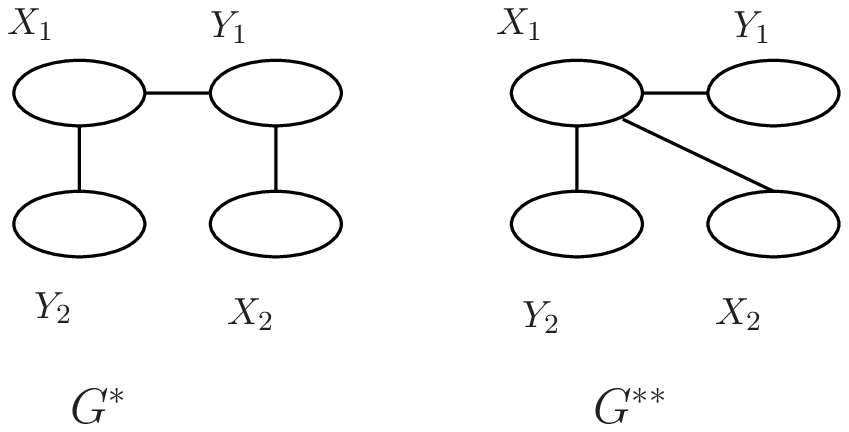}}\\[15pt]

Figure~1. $G^*$ and $G^{**}$
\end{center}
\end{figure}


\begin{lemma}
Let $G^*$ and $G^{**}$ be the graph defined above (see Figure 1).
Then one has
\begin{equation}\label{1.2}
EE(G^*)\leq EE(G^{**}),
\end{equation}
with equality holds if and only if $G^*\cong G^{**}$.
\end{lemma}

\begin{proof} Let $G_1=G^*[X_1\cup Y_2]$, $G_2=G^*[X\cup Y_1]$, and
$G'_1=G^{**}[X_1\cup Y_2]$, $G_2'=G^{**}[X\cup Y_1]$. We can see that
$G_1 = G'_1$, $G_2\cong K_{|X_1|+|X_2|, |Y_1|}$ and $G_2'\cong
K_{|X_1|, |Y_1|+|X_2|}$. Furthermore, $G^*=G_1\cup_{|X_1|}G_2$, and
$G^{**}=G'_1\cup_{|X_1|}G_2'$ with $S_1=S_2=S_1'=S_2'=X_1=\{x_1,
x_2\ldots, x_{|X_1|}\}$.

By Lemma \ref{lem:2.5}, we have
$$
M_{2k}(G_2)=2(|X_1|+|X_2|)^k|Y_1|^k \ , \ \ \
M_{2k}(G'_2)=2|X_1|^k(|X_2|+|Y_1|)^k.
$$
Since $|X_1|\geq |Y_1|$, we have $M_{2k}(G_2)\leq M_{2k}(G'_2)$.
Furthermore, as both $G_2$ and $G'_2$ are bipartite graphs, one has
$M_{2k-1}(G_2)=M_{2k-1}(G'_2)=0$ for any $k>0$. Now condition 1 of
Lemma \ref{lem:2.6} is satisfied.

For any $x_i, x_j\in X_1$, by Lemma \ref{lem:2.5} we know that for
any $l>0$,
\begin{eqnarray*}
  M_{2l}(G_2; x_i, x_j) &=& (|X_1|+|X_2|)^{l-1}|Y_1|^l=|Y_1|((|X_1|+|X_2|)|Y_1|)^{l-1}, \\
  M_{2l}(G_2'; x_i, x_j) &=& |X_1|^{l-1}(|X_2|+|Y_1|)^l=(|X_2|+|Y_1|)(|X_1|(|X_2|+|Y_1|))^{l-1}.
\end{eqnarray*}
As $|X_1|\geq |Y_1|$, we have $(|X_1|+|X_2|)|Y_1|\leq
|X_1|(|X_2|+|Y_1|)$. Hence
$$M_{2l}(G_2; x_i, x_j)\leq M_{2l}(G_2'; x_i, x_j),$$
with equality holds if and only if $|X_2|=0$. Together with
$M_{2l-1}(G_2; x_i, x_j)=$ $M_{2l-1}(G_2'; x_i, x_j)$ $=0$,
condition 2 of Lemma \ref{lem:2.6} is satisfied. So we have
$EE(G^*)\leq EE(G^{**})$, with equality holds if and only if
$|X_2|=0$, i.e., $G^*\cong G^{**}$.
\end{proof}

By (\ref{1.1}) and (\ref{1.2}), together with Corollary
\ref{coro:2.3}, it is straightforward to see that
\begin{theo}
Among the graphs in $\mathcal{M}_{n,p}$,  $K_{p,n-p}$ is the unique graph with maximum Estrada index.
\end{theo}

\section{Maximum Estrada index of  bipartite graphs with a given
connectivity(resp. edge connectivity)}

\hspace{8mm}  For two complete bipartite graphs $K_{n_1,n_2}$ and
$K_{m_1,m_2}$, we define a graph $O_s \vee_1(K_{n_1,n_2} \cup
K_{m_1,m_2} )$, where $\cup$ is the union of two graphs, $O_s \
(s\geq 1) $ is an empty graph of order $s$ and $\vee_1$ is a graph
operation that joins all the vertices in $O_s$ to the vertices
belonging to the partitions of cardinality $n_1$ in $K_{n_1,n_2}$
and $ m_1$ in $K_{m_1,m_2}$ (see Figure. 2), respectively.

\begin{figure}[h,t,b,p]
\begin{center}
\scalebox{0.8}[0.8]{\includegraphics{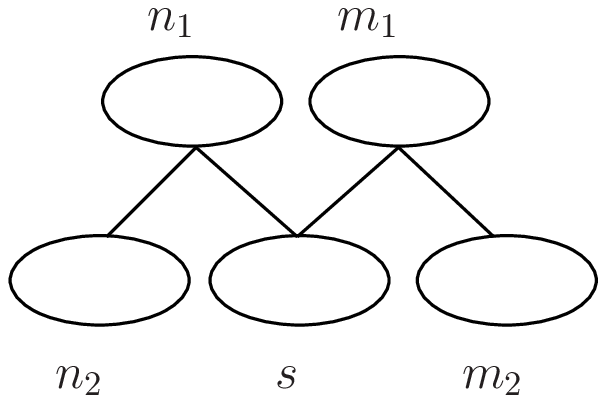}}\\[15pt]

Figure~2. $O_s \vee_1(K_{n_1,n_2} \cup K_{m_1,m_2})$
\end{center}
\end{figure}

\begin{lemma}\label{lem:4.1}
For an $n$-vertex bipartite graph $O_s \vee_1(K_1\cup K_{p,q})$ with
$p< q+s$ and $q\geq 0$, one has $EE(O_s \vee_1(K_1\cup K_{p,q})) <
EE(O_s \vee_1(K_1\cup K_{q+s,p-s}))$.
\end{lemma}
\begin{proof} Let us denote $O_s \vee_1(K_1\cup K_{p,q})$ by $G$ and
$O_s \vee_1(K_1\cup K_{q+s,p-s})$ by $G'$.

Let $G_1=G[\{u\}\cup O_s]$, $G_2=G-\{u\}$, we can see that
$G=G_1\cup_sG_2$. Similarly, let $G_1'=G'[\{u\}\cup O_s]$,
$G_2'=G'-\{u\}$, then $G'=G_1'\cup_sG_2'$.

It is obvious that $G_1\cong G_1'\cong K_{1,s}$, $G_2\cong
K_{s+q,p}$ and $G_2'\cong K_{s+p-s, q+s}$. Thus for any $k>0$,
$M_k(G_1)=M_k(G_1')$ and for any $1\leq i, j\leq s$, $(G_1;
a_i,a_j)\preceq (G_1'; a_i,a_j)$. From Lemma \ref{lem:2.5}, we know
that $M_k(G_2)=M_k(G_2')$ for any $k>0$.

Moreover, by Lemma \ref{lem:2.5} we have that for any $k>0$ and
$1\leq i, j\leq s$,
$$
M_{2k}(G_2; a_i,a_j)=(s+q)^{k-1}p^k<
(s+q)^{k}p^{k-1}=M_{2k}(G_2'; a_i,a_j).
$$
Together with $M_{2k-1}(G_2; a_i,a_j)=0=M_{2k-1}(G_2'; a_i,a_j)$, we
have $(G_2; a_i,a_j)\prec(G_2'; a_i,a_j)$. Hence, by Lemma
\ref{lem:2.6} we have $EE(G)< EE(G')$, as desired.
\end{proof}

\begin{lemma}\label{lem:4.2}
For an $n$-vertex bipartite graph $O_s \vee_1(K_1\cup K_{p,q})$ with
$p>q+s+1$ and $q>0$, one has $EE(O_s \vee_1(K_1\cup K_{p,q})) <
EE(O_s \vee_1(K_1\cup K_{p-1,q+1}))$.
\end{lemma}
\begin{proof}
Let $X=(x_1,x_2\ldots,x_n)^T$ be an eigenvector of $O_s
\vee_1(K_1\cup K_{p,q})$ corresponding to the eigenvalue $\lambda$.
By the eigenvalue-equations, for any $i=1,2,\ldots,n$,
$$
\lambda x_i=\sum_{j:v_i v_j\in E(O_s \vee_1(K_1\cup K_{p,q}))}x_j.
$$
Thus, for any eigenvalue of $O_s \vee_1(K_1\cup K_{p,q})$ with
$\lambda\not=0$, one has $x_i=x_j$ if $N(v_i)=N(v_j)$. So, we know
that the eigenvalue of $O_s \vee_1(K_1\cup K_{p,q})$ which is not
equal to 0 satisfies:
\begin{equation}\label{a}
  \left\{
   \begin{aligned}
   \lambda x_1 &= sx_2, \\
  \lambda x_2 &= x_1+px_3, \\
  \lambda x_3 &= sx_2+qx_4, \\
  \lambda x_4 &= px_3.
   \end{aligned}
   \right.
  \end{equation}
As the root of (\ref{a}) is also the root of
\begin{equation}\label{5}
  \lambda^4 - \lambda^2(s + pq + ps)  + pqs =0,
\end{equation}
then we have that
$$
EE(G)=n-4+2cosh(x_1)+2cosh(x_2),
$$
where $x_1, x_2$ are the different positive roots of (\ref{5}). We
may assume that $r=x_1>x_2$ and $k=x_1x_2=\sqrt{pqs}$. Then
$r>\sqrt{k}>0$, and we can get
\begin{equation}\label{6}
  EE(O_s \vee_1(K_1\cup K_{p,q}))=f(r,k)=n-4+2cosh(r)+2cosh(k/r).
\end{equation}
Then we have
\begin{equation}\label{7}
  \frac{\partial f(r,k)}{\partial r}=(e^r-e^{-r})-\frac{k}{r^2}(e^{k/r}-e^{-k/r})>0,
\end{equation}
and
\begin{equation}\label{8}
  \frac{\partial f(r,k)}{\partial k}=\frac{1}{r}(e^{k/r}-e^{-k/r})>0.
\end{equation}
Let $k'=\sqrt{(p-1)(q+1)s}$. As $pqs-(p-1)(q+1)s=s(q+1-p)<0,$ we
have $k<k'$.

On the other hand, let
$$
g(x,p,q,s)=x^4 - x^2(s + pq + ps)  + pqs,
$$
and $r'$ be the maximum root of $g(x,p-1,q+1,s)$ , we will show
$r'>r$. In fact, as
$g(r,p,q,s)-g(r,p-1,q+1,s)=(p-q-s-1)r^2-(p-s-1)s\geq
r^2-(p-s-1)s>0,$ we have $g(r,p-1,q+1,s)<0$. Together with
$g(\infty,p-1,q+1,s)>0$, we can get $r'>r$. Thus, by (\ref{7}) and
(\ref{8}) we have $f(r,k)<f(r',k')$, i.e., $EE(O_s \vee_1(K_1\cup
K_{p,q})) < EE(O_s \vee_1(K_1\cup K_{p-1,q+1}))$.
\end{proof}

\begin{lemma}\label{lem:4.2'}
For $s\leq \lceil\frac{n-1}{2}\rceil-1$, one has
$EE(K_{s,n-s}) < EE(O_s \vee_1(K_1\cup K_{n-s-2,1}))$.
\end{lemma}
\begin{proof}
By Lemma \ref{lem:2.2}, we have
$EE(K_{s,n-s})=n-2+2cosh(\sqrt{s(n-s)})$. As in the proof of Lemma
\ref{lem:4.2}, one has
$$
EE(O_s \vee_1(K_1\cup K_{n-s-2,1}))=n-4+2cosh(x_1)+2cosh(x_2),
$$
where $x_1$ and $x_2$ are the different positive roots of $f(x)=0$,
where
$$
f(x)= x^4 - x^2(s + (n-s-2)+(n-s-2)s) +(n-s-2)s.
$$
Without loss of generality, we assume that $x_1>x_2$. Then we have
\begin{eqnarray*}
  f(\sqrt{s(n-s)}) &=& -s(n^2 - 3ns - 3n + 2s^2 + 3s + 2) \\
                   &=& -s((n-2s-3)(n-s)+2)<0,
\end{eqnarray*}
where the $``<"$ holds since $s\leq \lceil\frac{n-1}{2}\rceil-1$,
i.e., $n\geq 2s+3$. Together with $f(\infty)>0$, we have
$x_1>\sqrt{s(n-s)}$.

Now $x_1>\sqrt{s(n-s)}, x_2>0$, then by the monotonicity of
$cosh(x)$, one has $cosh(x_1)>cosh(\sqrt{s(n-s)})$ and
$cosh(x_2)>1$. We then deduce
$$EE(K_{s,n-s})<EE(O_s \vee_1(K_1\cup K_{n-s-2,1})),$$
as desired.
\end{proof}

\begin{lemma}\label{lem:4.3}
Let  $G$ be a graph with maximum Estrada index in
$\mathcal{C}_{n,s}$ and $U$ be a minimum vertex cut. If $G-U$ has a
nontrivial component $G_1$, then $G-U$ has exactly two components,
and the other component which is distinct from $G_1$ cannot be
nontrivial.
\end{lemma}

\begin{proof} Let $G_1,G_2,\ldots,G_k$ be the components of $G-U$.
Suppose $k\geq 3$. Then, we can add some appropriate edges in $G$
between $G_1,G_2,\ldots,G_{k-1}$ so that the resulting graph $G'$ is
still bipartite. It is obvious that $G'\in \mathcal{C}_{n,s}$.  By
Lemma \ref{lem:2.1}, we have $EE(G')>EE(G)$. This contradicts the
fact that $G$ has the maximum Estrada index among graphs in
$\mathcal{C}_{n,s}$, and so we have $k=2$.

If both $G_1$ and $G_2$ are nontrivial with bipartition $(A, B)$ and
$(C, D)$, respectively. Let $U = U_1 \cup U_2$ be the bipartition of
$U$ induced by the bipartition of $G$. Now joining all possible
edges between the vertices of $A$ and $B$, $C$ and $D$, and $U_1$
and $U_2$, we get a graph $\widehat{G}$ in $\mathcal{C}_{n,s}$ such
that $EE(\widehat{G})\geq EE(G)$. Therefore, we assume
$G=\widehat{G}$; see Figure 3. Suppose that $G_1$ and $G_2$ are the
two nontrivial components of $G-U$. The bipartition of $G_1$ is $(A,
B)$ and the bipartition of $G_2$ is $(C, D)$. Let $U = U_1 \cup U_2$
be the bipartition of $U$ induced by the bipartition of $G$. Now
joining all possible edges between the vertices of $A$ and $B$, $C$
and $D$, and $U_1$ and $U_2$, we get a graph $\widehat{G}$ in
$\mathcal{C}_{n,s}$ such that $EE(\widehat{G})\geq EE(G)$.
Therefore, we assume $G=\widehat{G}$; see Figure 3.
\begin{figure}[h,t,b,p]
\begin{center}
\scalebox{1}[1]{\includegraphics{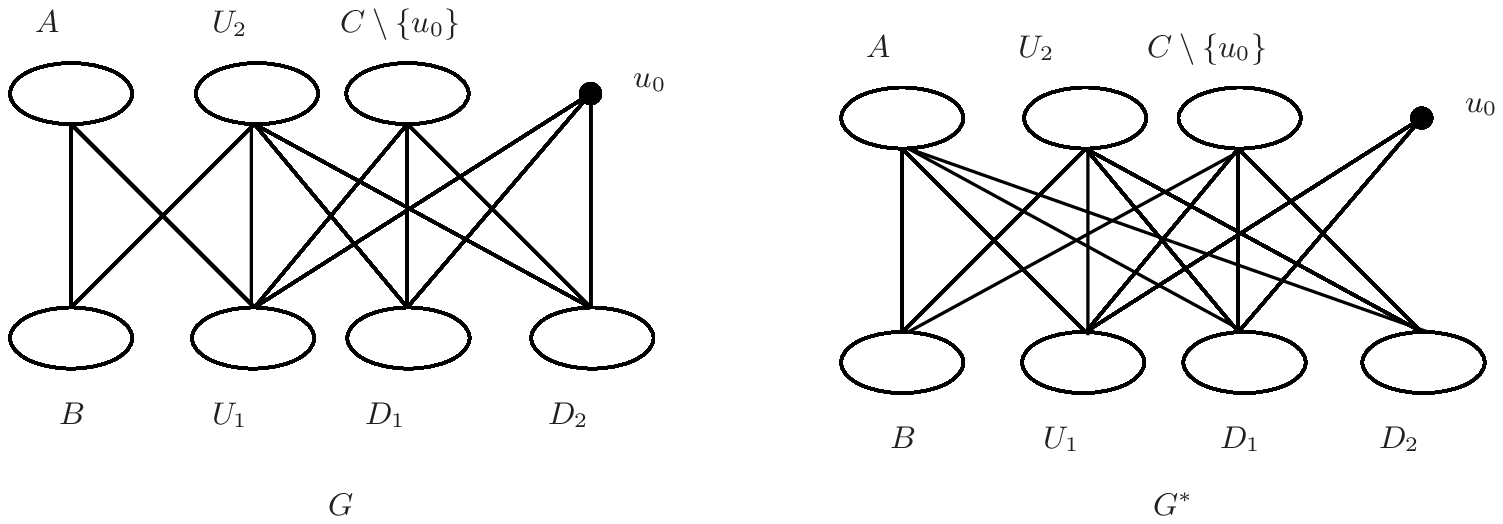}}\\[15pt]

Figure~3. $G^*$ and $G^{**}$
\end{center}
\end{figure}

If there exists some vertex $w$ in $G-U$ such that $d_G(w)= s$, then
forming a complete bipartite graph within the vertices of $G
\setminus{w}$ we would get a graph in $\mathcal{C}_{n,s}$  with
larger  Estrada index. Thus, we may assume that each vertex in $G-U$
has a degree greater than $s$. Let $|A| = m_1, \ |B| = m_2,\  |C| =
n_1,\  |D| = n_2, \ |U_1| = t, \ |U_2| = k.$

We choose a vertex $u_0$ from $C$ and observe that $d_G(u_0) = t +
|D| > s$, where $t \ (0\leq t \leq s)$ is the total number of edges
joining $u_0$ and the vertices of $U_1$. Note that $U_1 \cup  U_2$
is the vertex cut of order $s$, hence $m_1, n_1 > t, \ m_2, n_2
> k$. Without loss of generality, we may assume that $m_1 = \max
\{m_1,m_2, n_1, n_2\}$ and note that $s\geq 1$, hence $m_1 \geq 2$.
We now choose a subset $D_2$ of $D$ such that $|D_2| = |D|-k > 0$.
Let
$$G^*= G-\{u_0x : x \in D_2\} + \{bc : b\in  B, c \in C \setminus
\{u_0\}\} + \{pq : p\in D, c\in  A\}.$$
It is routine to check that $G^*\in  \mathcal{C}_{n,s}$ with
bipartition $(X, Y)$. We claim that $EE(G)< EE(G^*).$

We color the edges $u_0x$ blue if $x\in D_2$ and red if $x\in
U_1\cup D_1$. For any $k>0$ and $W\in  W_{2k}(G)$, let $\Psi(W)$ be
the closed walk of length $2k$ that is obtained by changing all the
blue edges $u_0x$ to $ax$ and all the red edges $u_0y$ which is
incident with a red edge to $ay$, where $a$ is a vertex of $A$. It
is obvious that $\Psi(W)\in W_{2k}(G^*)$ and $\Psi$ is an injection.
Hence we have for any $k> 0$, $M_{2k}\leq M_{2k}(G^*)$. Together
with $M_{2k-1}=M_{2k-1}(G^*)=0$, we get $EE(G)\leq EE(G^*)$.
Furthermore, as $M_2(G)=|E(G)|<|E(G^*)|$, we know that $EE(G)<
EE(G^*)$. So, we get our conclusion.
\end{proof}

\begin{theo}\label{thm:4.3}
The unique graph in $\mathcal{C}_{n,s}$ with the maximum Estrada
index is $O_s \vee_1(K_1\cup K_{\lfloor\frac{n-1}{2}\rfloor,
\lceil\frac{n-1}{2}\rceil-s})$.
\end{theo}

\begin{proof}
Let $G$ be a graph with the maximum Estrada index in $C_{s,n}$. Let
$U$ be a vertex cut of $G$ containing $s$ vertices. we distinguish
the following two cases:

Case 1. All the components of $G-U$ are singletons. In this case, we
have $G = K_{s,n-s}$. For $s = \lceil\frac{n-1}{2}\rceil$, it is
nothing to say since $K_{s,n-s}\cong O_s \vee_1(K_1\cup
K_{\lfloor\frac{n-1}{2}\rfloor, \lceil\frac{n-1}{2}\rceil-s})$. For
$1\leq s\leq \lceil\frac{n-1}{2}\rceil-1$, by Lemma \ref{lem:4.2'},
$EE(K_{s,n-s})< EE(O_s \vee_1(K_1\cup K_{n-s-2,1}))$, which
contradicts the maximality of $G$.

Case 2. One component of $G-U$, say $G_1$, contains at least two
vertices. By Lemma \ref{lem:4.3}, we know that  $G-U$ has exactly
two components $G_1$ and $G_2$, with $G_2\cong K_1$. Therefore,
there exist $p,q$ with $p+q+s+1=n, p\geq s, q>0$, such that $G\cong
O_s \vee_1(K_1\cup K_{p,q})$. Then by Lemma \ref{lem:4.1} and Lemma
\ref{lem:4.2}, we have $q+s\leq p\leq q+s+1, p+q+s+1=n$. Hence we
have $p=\lfloor\frac{n-1}{2}\rfloor$, i.e., $G\cong O_s
\vee_1(K_1\cup K_{\lfloor\frac{n-1}{2}\rfloor,
\lceil\frac{n-1}{2}\rceil-s})$. This completes the proof.
\end{proof}
\begin{coro}
Let $\mathcal{D}_{n,s}$ be the set of graphs with $n$ vertices and
edge-connectivity $s$. One has that the unique graph in
$\mathcal{D}_{n,s}$ with the maximum Estrada index is $O_s
\vee_1(K_1\cup K_{\lfloor\frac{n-1}{2}\rfloor,
\lceil\frac{n-1}{2}\rceil-s})$.
\end{coro}
\begin{proof}
As is well known \cite{Bondy} that $\kappa(G)\leq \kappa'(G)$, hence
for any $G\in \mathcal{D}_{n,s}$ there exists $k\leq s$ such that
$G\in \mathcal{C}_{n,k}$. Then by Theorem \ref{thm:4.3}, we have
$EE(G)\leq EE(O_k \vee_1(K_1\cup K_{\lfloor\frac{n-1}{2}\rfloor,
\lceil\frac{n-1}{2}\rceil-k}))$. On the other hand, $O_s
\vee_1(K_1\cup K_{\lfloor\frac{n-1}{2}\rfloor,
\lceil\frac{n-1}{2}\rceil-s})\in \mathcal{D}_{n,s}$ and $O_k
\vee_1(K_1\cup K_{\lfloor\frac{n-1}{2}\rfloor,
\lceil\frac{n-1}{2}\rceil-k})\subset O_s \vee_1(K_1\cup
K_{\lfloor\frac{n-1}{2}\rfloor, \lceil\frac{n-1}{2}\rceil-s})$ as
$k\leq s$, and hence we get our conclusion.
\end{proof}

\end{document}